\documentclass[11pt]{article}
\usepackage[margin=1.12in]{geometry}
\usepackage[T1]{fontenc}
\usepackage{lmodern}
\usepackage{microtype}
\usepackage{amsmath,amssymb,amsthm,mathtools}
\usepackage{aliascnt}
\usepackage[hidelinks]{hyperref}
\usepackage[nameinlink,noabbrev]{cleveref}

\hypersetup{
  pdftitle={Augmentation-Ideal Realization of Special Nottingham Congruence Series and Cartier Indecomposables},
  pdfauthor={Chao Ma},
  pdfcreator={Chao Ma},
  pdfsubject={Realization of the shifted congruence series of a special Nottingham group by the canonical weighted augmentation filtration, and the Cartier description of the Lie indecomposables of its group-symbol algebra},
  pdfkeywords={augmentation ideal; N-series; dimension subgroup; Nottingham jet group; volume-preserving automorphism; divergence-free polynomial field; Cartier operator; polynomial de Rham cohomology; Frobenius twist; group ring; filtered representation}
}

\numberwithin{equation}{section}
\newtheorem{theorem}{Theorem}[section]
\newaliascnt{maintheorem}{theorem}
\newtheorem{maintheorem}[maintheorem]{Main Theorem}
\aliascntresetthe{maintheorem}
\newaliascnt{proposition}{theorem}
\newtheorem{proposition}[proposition]{Proposition}
\aliascntresetthe{proposition}
\newaliascnt{lemma}{theorem}
\newtheorem{lemma}[lemma]{Lemma}
\aliascntresetthe{lemma}
\newaliascnt{corollary}{theorem}
\newtheorem{corollary}[corollary]{Corollary}
\aliascntresetthe{corollary}
\newaliascnt{definition}{theorem}
\newtheorem{definition}[definition]{Definition}
\aliascntresetthe{definition}
\newaliascnt{remark}{theorem}

\aliascntresetthe{remark}

\crefname{maintheorem}{Main Theorem}{Main Theorems}
\Crefname{maintheorem}{Main Theorem}{Main Theorems}
\crefname{theorem}{theorem}{theorems}
\Crefname{theorem}{Theorem}{Theorems}
\crefname{proposition}{proposition}{propositions}
\Crefname{proposition}{Proposition}{Propositions}
\crefname{lemma}{lemma}{lemmas}
\Crefname{lemma}{Lemma}{Lemmas}
\crefname{corollary}{corollary}{corollaries}
\Crefname{corollary}{Corollary}{Corollaries}
\crefname{definition}{definition}{definitions}
\Crefname{definition}{Definition}{Definitions}
\crefname{remark}{remark}{remarks}
\Crefname{remark}{Remark}{Remarks}

\DeclareMathOperator{\Sym}{Sym}
\DeclareMathOperator{\End}{End}
\DeclareMathOperator{\gr}{gr}
\DeclareMathOperator{\diver}{div}
\DeclareMathOperator{\Span}{span}

\newcommand{\m}{\mathfrak m}
\newcommand{\cI}{\mathcal I}
\newcommand{\cA}{\mathcal A}
\newcommand{\cE}{\mathcal E}
\newcommand{\cL}{\mathcal L}
\newcommand{\cX}{\mathcal X}
\newcommand{\Ld}{\mathsf L}
\newcommand{\Ex}{\mathsf E}
\newcommand{\vol}{\operatorname{vol}}
\newcommand{\one}{\mathbf 1}
\newcommand{\dR}{\mathrm{dR}}

\title{Augmentation-Ideal Realization of Special Nottingham\\
Congruence Series and Cartier Indecomposables}

\author{Chao Ma\thanks{Independent Researcher, London, United Kingdom.
Email: \href{mailto:raymond.ma@me.com}{\texttt{raymond.ma@me.com}}.
ORCID: \href{https://orcid.org/0009-0004-2456-9098}
{0009-0004-2456-9098}.}}
\date{}

\begin{document}
\maketitle

\begin{abstract}
The shifted congruence series of finite special Nottingham jet groups is
realized by an integral augmentation filtration.  For volume-preserving
formal substitutions in characteristic $p\ge5$ and $n\ge3$ variables, the
filtration induced by the action on truncated power series has the
congruence subgroups as its dimension subgroups.  The canonical weighted
augmentation filtration has the same dimension subgroups.  The group
symbols form a degree-truncated Lie algebra of divergence-free polynomial
fields.  Since brackets with quadratic fields give the subspace of exact
fields in every coefficient degree at least three, the Lie indecomposables occur only in degrees $(p-1)(n-1)+pr$, with
$r\ge0$, where they are determinant-weighted Frobenius twists of tensor
products of symmetric powers with the defining module.
\end{abstract}

\section{Introduction}

We realize the shifted congruence series of finite special Nottingham jet
groups by an integral augmentation filtration and determine the exceptional
graded Lie indecomposables of its group symbols.
Let $k=\mathbf F_p$, let
\[
 A=k[[x_1,\ldots,x_n]],\qquad \m=(x_1,\ldots,x_n),
\]
and fix the volume form
\[
 \vol=dx_1\wedge\cdots\wedge dx_n.
\]
Let $G$ be the group of continuous $k$-algebra automorphisms of $A$ that
are tangent to the identity and preserve $\vol$.  Tangency to the identity
means $g(x_j)-x_j\in\m^2$ for every $j$; accordingly we also write $G_2=G$.
For $s\ge2$, put
\[
 G_s=\{g\in G:g(x_j)-x_j\in\m^s\text{ for every }j\}.
\]
For a fixed $N\ge3$, write
\[
 T_N=G/G_N=G_2/G_N.
\]
It is a finite $p$-group: the congruence filtration has finitely many
finite-dimensional $k$-vector-space quotients $G_s/G_{s+1}$ (as established
below), and terminates at $G_N/G_N$.
We write $G_s/G_N$ for the image $G_sG_N/G_N$ in $T_N$; it is trivial
for $s\ge N$.  For $a\ge1$, put
\[
 H_a=G_{a+1}/G_N.
\]
Then
\[
 T_N=H_1\supseteq H_2\supseteq\cdots\supseteq H_{N-1}=1
\]
is an $N$-series:
\[
 [H_a,H_b]\subseteq H_{a+b}.
\]
The dimension-subgroup approach to filtered groups originates in work of
Magnus, Zassenhaus, Jennings, and Lazard
\cite{Magnus1937,Zassenhaus1939,Jennings1941,Lazard1953,Lazard1954}.
Quillen later identified the associated graded group algebra with an
enveloping algebra in characteristic $p$~\cite{Quillen1968}.  Losey
attributed to Lazard the realization problem for a prescribed
$N$-series~\cite{Losey1974}.  A general affirmative answer would imply the
classical dimension-subgroup conjecture, which Rips disproved with a finite
$2$-group counterexample~\cite{Rips1972}.  The congruence series considered here admits a canonical realization.  Hartl studied related group-ring
filtrations induced by $N$-series~\cite{Hartl1995}.

J\k{e}drzejewicz described divergence-free polynomial derivations in
positive characteristic~\cite{Jedrzejewicz2019}.  In the divided-power
Cartan-type setting, Benkart, Gregory, and Premet determined the homogeneous
divergence-free modules in natural degrees $\ell\le p-2$
\cite[Theorem~2.69]{BGP2009}, corresponding here to coefficient degrees
$d\le p-1$.  For polynomial fields over every field,
$[\Ld_2,\Ld_{d-1}]=\Ex_d$ for every $d\ge3$ \cite[Theorem~A]{MaDerived2026}.
Cartier's operator on polynomial differential forms then determines the
exceptional quotient \cite{Cartier1957}.

Congruence depth in $T_N$ is detected by its action on the truncated local
algebra
\[
 B_N=A/\m^N,
 \qquad F^qB_N=(\m^q+\m^N)/\m^N\quad(q\ge0).
\]
The action of $T_N$ on $B_N$ preserves the $\m$-adic filtration.  Let
$\cI_a\subseteq\Delta_{\mathbf Z}(T_N)$ be the pullback of the endomorphisms
that raise this filtration by at least $a$.  The substitution estimate yields
\[
 g\in G_{a+1}/G_N
 \quad\Longleftrightarrow\quad
 (g-1)(F^qB_N)\subseteq F^{q+a}B_N\quad\text{for all }q
\]
which determines the corresponding dimension subgroups.

The canonical weighted augmentation filtration of the series
$H_\bullet$ is
\[
 \cA_a(H_\bullet)=
 \Bigl\langle
 (h_1-1)\cdots(h_t-1):
 h_i\in H_{b_i},\ \sum_i b_i\ge a
 \Bigr\rangle_{\mathbf Z}.
\]
The operator filtration contains $\cA_a(H_\bullet)$ in every degree; the
canonical weighted filtration therefore has the same dimension subgroups.

\begin{maintheorem}[Augmentation realization]
\label{thm:augmentation-realization}
For every $N\ge3$, the canonical weighted filtration
\[
 \Delta_{\mathbf Z}(T_N)=\cA_1(H_\bullet)
 \supseteq\cA_2(H_\bullet)\supseteq\cdots
\]
is a multiplicative filtration by two-sided ideals and satisfies
\[
 \mathfrak D\bigl(\cA_a(H_\bullet)\bigr)=H_a=G_{a+1}/G_N
 \qquad(a\ge1).
\]
The jet action supplies a second multiplicative filtration
\[
 \Delta_{\mathbf Z}(T_N)=\cI_1\supseteq\cI_2\supseteq\cdots
\]
by two-sided ideals such that
\[
 \cA_a(H_\bullet)\subseteq\cI_a,
 \qquad
 \mathfrak D(\cI_a):=\{g\in T_N:g-1\in\cI_a\}=G_{a+1}/G_N
\]
for every $a\ge1$.
Write $\cI_a^N$ when the dependence on $N$ matters.  If $N'\ge N$, the
quotient homomorphism
$\mathbf Z[T_{N'}]\to\mathbf Z[T_N]$ sends $\cI_a^{N'}$ into
$\cI_a^N$.
\end{maintheorem}

For $2\le s<N$, the symbols of elements of $G_s/G_N$ form a subspace
\[
 \cL_s\subseteq\cI_{s-1}/\cI_s.
\]
Place $\cL_s$ in graded degree $s-1$.  The multiplication commutator then
makes $\bigoplus_s\cL_s$ a graded Lie algebra.  Via leading fields, it is
identified with the similarly shifted truncation of
\[
 \Ld_s=\left\{\sum_i f_i\partial_i:
 f_i\in k[x_1,\ldots,x_n]_s,\ \sum_i\partial_i f_i=0\right\}.
\]
Let $\Ex_s\subseteq\Ld_s$ correspond, under contraction with $\vol$, to
the exact $(n-1)$-forms.

\begin{maintheorem}[Cartier augmentation quotient]
\label{thm:cartier-augmentation}
For every $3\le d<N$,
\[
 [\cL_2,\cL_{d-1}]=\cX_d,
\]
where $\cX_d\subseteq\cL_d$ corresponds to $\Ex_d$.  Hence
\[
 \cL_d/[\cL_2,\cL_{d-1}]\cong\Ld_d/\Ex_d.
\]
Equivalently, this is the degree-$d$ component of the Lie abelianization
\[
 Q(\cL_{<N})=\cL_{<N}/[\cL_{<N},\cL_{<N}].
\]
This quotient vanishes unless
\[
 d=d_r=(p-1)(n-1)+pr
\]
for an integer $r\ge0$.  If $d_r<N$, the leading-field identification gives
the quotient a rational $\mathrm{GL}(V)$-module structure and an isomorphism
\[
 \cL_{d_r}/[\cL_2,\cL_{d_r-1}]
 \cong
 \delta_1\otimes W_r^{(1)},
\]
where
\[
 V=k^n,\qquad
 W_r=\Sym^r(V^*)\otimes V,\qquad
 \delta_1=(\det V^*)^{\otimes(p-1)}.
\]
\end{maintheorem}

These quotients are stable under truncation in every fixed degree.

\section{Filtered representations and augmentation ideals}

Let $B$ be a finite-dimensional $k$-algebra with a finite decreasing
filtration
\[
 B=F^0B\supseteq F^1B\supseteq\cdots\supseteq F^MB=0,
\qquad
 F^iB\,F^jB\subseteq F^{i+j}B.
\]
We extend the filtration by $F^qB=0$ for $q\ge M$.
Let $\End_F(B)$ be the algebra of $k$-linear endomorphisms preserving this
filtration.

\begin{definition}[Operator filtration]
For $a\ge0$, set
\[
 \cE_a(B)=\{u\in\End_F(B):u(F^qB)\subseteq F^{q+a}B
 \text{ for all }q\ge0\}.
\]
\end{definition}

\begin{lemma}[Filtered endomorphism ideals]
\label{lem:endomorphism-ideals}
Each $\cE_a(B)$ is a two-sided ideal of $\End_F(B)$, and
\[
 \cE_{a+1}(B)\subseteq\cE_a(B),
 \qquad
 \cE_a(B)\cE_b(B)\subseteq\cE_{a+b}(B).
\]
\end{lemma}

\begin{proof}
If $u$ raises filtration by $a$ and $v$ preserves it, then both $vu$ and
$uv$ raise filtration by $a$.  If $u$ raises by $a$ and $w$ by $b$, then
$uw$ raises filtration by $a+b$.  Finally,
$F^{q+a+1}B\subseteq F^{q+a}B$ gives $\cE_{a+1}(B)\subseteq\cE_a(B)$.
\end{proof}

Let $\Gamma\le\operatorname{Aut}_{k\text{-alg}}(B)$ be a subgroup preserving
$F^\bullet B$.  We compose automorphisms as operators, so that $gh$ acts
as $g\circ h$, and use the convention
$[g,h]=g^{-1}h^{-1}gh$.  The action gives a ring homomorphism
\[
 \rho_{\mathbf Z}:\mathbf Z[\Gamma]\longrightarrow\End_F(B),
\]
where the integer coefficients act through $\mathbf Z\to k$.  Let
$\Delta_{\mathbf Z}(\Gamma)$ be the integral augmentation ideal.

\begin{definition}[Ideals defined by the action]
For $a\ge1$, define
\[
 \cI_a(\Gamma,B)=
 \Delta_{\mathbf Z}(\Gamma)\cap\rho_{\mathbf Z}^{-1}(\cE_a(B)).
\]
\end{definition}

\begin{proposition}[Pullback to the augmentation ideal]
\label{prop:operator-filtration}
Suppose that every $g\in\Gamma$ acts trivially on $\gr_F B$.  Then
\[
 \cI_1(\Gamma,B)=\Delta_{\mathbf Z}(\Gamma),
\]
each $\cI_a(\Gamma,B)$ is a two-sided ideal of
$\mathbf Z[\Gamma]$ contained in $\Delta_{\mathbf Z}(\Gamma)$, and
\[
 \cI_{a+1}(\Gamma,B)\subseteq\cI_a(\Gamma,B),
 \qquad
 \cI_a(\Gamma,B)\cI_b(\Gamma,B)
 \subseteq\cI_{a+b}(\Gamma,B).
\]
\end{proposition}

\begin{proof}
Write an augmentation element as a finite sum
\[
 z=\sum_g n_g(g-1).
\]
Because $g$ acts trivially on $\gr_F B$, each operator $g-1$ raises
filtration by at least one.  Thus $\rho_{\mathbf Z}(z)\in\cE_1(B)$ and
$\cI_1=\Delta_{\mathbf Z}(\Gamma)$.

Let $x\in\cI_a$ and $y\in\mathbf Z[\Gamma]$.  Since $\cE_a(B)$ is a
two-sided ideal of $\End_F(B)$, both $\rho(yx)$ and $\rho(xy)$ lie in
$\cE_a(B)$.  Since $x$ has augmentation zero, so do $yx$ and $xy$.  Hence
$yx,xy\in\cI_a$, and $\cI_a$ is a two-sided ideal of the full group ring.
The descending and multiplicative inclusions are inherited from the operator
filtration.
\end{proof}

For an ideal $J\subseteq\Delta_{\mathbf Z}(\Gamma)$, define its dimension
subgroup by
\[
 \mathfrak D(J)=\{g\in\Gamma:g-1\in J\}.
\]

\begin{theorem}[Dimension subgroups from a filtered action]
\label{thm:operator-realization}
Under the hypotheses of \Cref{prop:operator-filtration}, set
\[
 \Gamma_a=\{g\in\Gamma:(g-1)(F^qB)\subseteq F^{q+a}B
 \text{ for all }q\}.
\]
Then
\[
 \mathfrak D(\cI_a(\Gamma,B))=\Gamma_a
\]
for every $a\ge1$.
\end{theorem}

\begin{proof}
For $g\in\Gamma$,
\[
 g-1\in\cI_a
 \quad\Longleftrightarrow\quad
 \rho(g)-1\in\cE_a(B)
 \quad\Longleftrightarrow\quad
 g\in\Gamma_a.
\]
\end{proof}

\section{Special Nottingham jet groups}

Fix $p\ge5$ and $n\ge3$.  Put
\[
 A=k[[x_1,\ldots,x_n]],\qquad
 B_N=A/\m^N,
\qquad F^qB_N=(\m^q+\m^N)/\m^N\quad(q\ge0).
\]
Every element of $T_N$ preserves $F^\bullet B_N$ and acts trivially on
$\gr_F B_N$.
For $g\in G$, write $g^*$ for the induced operator on $A$; thus
$(gh)^*=g^*h^*$.

\begin{lemma}[Substitution estimate]
\label{lem:substitution-estimate}
Let $g\in G_s$ with $s\ge2$.  Then for every $q\ge0$,
\[
 (g^*-1)(\m^q)\subseteq\m^{q+s-1}.
\]
Conversely, if
\[
 (g^*-1)(\m)\subseteq\m^s,
\]
then $g\in G_s$.
\end{lemma}

\begin{proof}
Write
\[
 g^*(x_i)=x_i+h_i,\qquad h_i\in\m^s.
\]
For a monomial $x^\alpha$ of degree $q$, expand
\[
 g^*(x^\alpha)-x^\alpha
 =\prod_i(x_i+h_i)^{\alpha_i}-\prod_i x_i^{\alpha_i}.
\]
Every term in the difference contains some $h_i$ in place of an $x_i$ and
therefore has degree at least $q-1+s$.  Linearity and continuity give the
forward inclusion.  The converse follows by applying the assumption to the
coordinates $x_i\in\m$.
\end{proof}

\begin{proposition}[Congruence depth from the jet action]
\label{prop:operator-congruence}
For $a\ge1$ and $g\in T_N$,
\[
 (g^*-1)(F^qB_N)\subseteq F^{q+a}B_N\quad\text{for all }q
\]
if and only if $g\in G_{a+1}/G_N$.
\end{proposition}

\begin{proof}
Choose a representative $\widetilde g\in G$ of $g$.  Suppose first that the
operator condition holds.  If $a+1<N$, its case $q=1$ gives
\[
 \widetilde g^*(x_i)-x_i\in\m^{a+1}+\m^N=\m^{a+1}
 \qquad(1\le i\le n),
\]
so $\widetilde g\in G_{a+1}$.  If $a+1\ge N$, then
$F^{a+1}B_N=0$, and the same condition shows that $g$ fixes every
coordinate of $B_N$.  Hence $\widetilde g\in G_N$ and $g=1$.

Conversely, a representative in $G_{a+1}$ satisfies the required
inclusions by \Cref{lem:substitution-estimate}, after reduction modulo
$\m^N$.
\end{proof}

\begin{lemma}[Basic properties of the congruence filtration]
\label{lem:basic-depth-estimates}
For every $s\ge2$, the subgroup $G_s$ is normal in $G$.  The kernel of the
action of $G$ on $B_N=A/\m^N$ is exactly $G_N$, so the induced action of
$T_N=G/G_N$ is faithful.  Moreover, for $a,b\ge2$,
\[
 [G_a,G_b]\subseteq G_{a+b-1}.
\]
\end{lemma}

\begin{proof}
Every formal automorphism with invertible linear part preserves each
power $\m^q$.  If $h\in G_s$, then by \Cref{lem:substitution-estimate} the
operator $h^*-1$ raises the $\m$-adic filtration by at least $s-1$.
Conjugating this operator by $g^*$ and $(g^{-1})^*$, which preserve the
filtration, does not change that lower bound.  The converse part of
\Cref{lem:substitution-estimate} therefore gives
$ghg^{-1}\in G_s$, proving normality.

An element $g\in G$ acts trivially on $B_N$ exactly when
$g^*(x_i)-x_i\in\m^N$ for every $i$, that is, exactly when $g\in G_N$.
Hence the induced action of $T_N$ is faithful.

For the commutator estimate, write on $B_M$, for arbitrary $M$,
\[
 g^*=1+u,\qquad h^*=1+v,
\]
where $u$ raises filtration by $a-1$ and $v$ by $b-1$.  The exact identity
\[
 [g^*,h^*]-1=(g^*)^{-1}(h^*)^{-1}(uv-vu)
\]
shows that $[g^*,h^*]-1$ raises filtration by at least
$(a-1)+(b-1)=a+b-2$.  Since $M$ is arbitrary, the converse substitution
estimate gives $[g,h]\in G_{a+b-1}$.
\end{proof}

\begin{proof}[Proof of \Cref{thm:augmentation-realization}]
By \Cref{lem:basic-depth-estimates}, the action of $T_N$ on $B_N$ is
faithful.  Apply \Cref{prop:operator-filtration,thm:operator-realization}
to this action, and define
\[
 \cI_a=
 \Delta_{\mathbf Z}(T_N)\cap
 \rho_{\mathbf Z}^{-1}(\cE_a(B_N)).
\]
Then \Cref{prop:operator-congruence} gives
\[
 \mathfrak D(\cI_a)=G_{a+1}/G_N.
\]
The groups $\cA_a(H_\bullet)$ form a multiplicative filtration and
$\cA_1(H_\bullet)=\Delta_{\mathbf Z}(T_N)$.  They are two-sided ideals:
if $P$ is a weighted generator of $\cA_a(H_\bullet)$ and $g\in T_N=H_1$,
then
\[
 gP=P+(g-1)P,
 \qquad
 Pg=P+P(g-1),
\]
and the additional terms have weight at least $a+1$.

For $h\in H_b$, the equality $\mathfrak D(\cI_b)=H_b$ gives
$h-1\in\cI_b$.  Multiplicativity puts each generator
\[
 (h_1-1)\cdots(h_t-1),
 \qquad h_i\in H_{b_i},\quad \sum_i b_i\ge a,
\]
of $\cA_a(H_\bullet)$ in $\cI_a$.  Hence
\[
 \cA_a(H_\bullet)\subseteq\cI_a.
\]
For every $h\in H_a$, the element $h-1$ is a weight-$a$ generator of
$\cA_a(H_\bullet)$.  Hence
\[
 H_a\subseteq \mathfrak D\bigl(\cA_a(H_\bullet)\bigr)
 \subseteq \mathfrak D(\cI_a)=H_a,
\]
which proves the canonical realization assertion.

For $N'\ge N$, let
\[
 \pi_*:\mathbf Z[T_{N'}]\longrightarrow\mathbf Z[T_N]
\]
be induced by truncation.  If $z\in\cI_a^{N'}$, the operator induced by $\pi_*(z)$ on $B_N$ is the
quotient of the operator induced by $z$ on $B_{N'}$; it therefore still
raises the $\m$-adic filtration by at least $a$.  Hence
\[
 \pi_*(\cI_a^{N'})\subseteq\cI_a^N.
\]
\end{proof}

\begin{corollary}[The congruence $N$-series]
The sequence $H_a=G_{a+1}/G_N$ satisfies
\[
 [H_a,H_b]\subseteq H_{a+b}.
\]
\end{corollary}

\begin{proof}
Apply \Cref{lem:basic-depth-estimates} with indices shifted by one.
\end{proof}

\section{Group symbols in the augmentation associated graded}

The integral action on $B_N$ factors through characteristic $p$.  Thus, for
every $a,b\ge1$,
\[
 p\cI_a\subseteq\cI_b,
\]
because $\rho_{\mathbf Z}(pz)=0$ for every $z\in\cI_a$.  In particular, $p\cI_a\subseteq\cI_{a+1}$.  Hence each quotient
$\cI_a/\cI_{a+1}$ is a $k=\mathbf F_p$-vector space and
$\gr_{\cI}\Delta_{\mathbf Z}(T_N)$ is a graded $k$-algebra.

For $2\le s<N$, define the group-symbol subspace
\[
 \cL_s=
 \{(g-1)+\cI_s:g\in G_s/G_N\}
 \subseteq \cI_{s-1}/\cI_s.
\]
\begin{lemma}[Additive symbol map]
\label{lem:additive-symbol}
The map
\[
 \sigma_s:G_s/G_{s+1}\longrightarrow\cL_s,
 \qquad gG_{s+1}\longmapsto(g-1)+\cI_s,
\]
is a well-defined injective homomorphism of elementary abelian $p$-groups,
and its image is $\cL_s$.
\end{lemma}

\begin{proof}
The kernel is
\[
 \{g\in G_s:g-1\in\cI_s\}=G_{s+1}
\]
by \Cref{thm:augmentation-realization}.  If $g,h\in G_s$, then
\[
 gh-1=(g-1)+(h-1)+(g-1)(h-1).
\]
The product belongs to $\cI_{2s-2}\subseteq\cI_s$ because $s\ge2$.
Hence $\sigma_s(gh)=\sigma_s(g)+\sigma_s(h)$.  Since the target is a $k$-vector space, injectivity shows that
$G_s/G_{s+1}$ is elementary abelian.  Its image is $\cL_s$ by definition.
\end{proof}

\begin{lemma}[Compatibility of brackets]
\label{lem:bracket-symbol}
If $g\in G_a$ and $h\in G_b$, then the associative commutator of their
augmentation symbols equals the group-commutator symbol:
\[
 [(g-1)+\cI_a,(h-1)+\cI_b]
 =([g,h]-1)+\cI_{a+b-1}
\]
in $\cI_{a+b-2}/\cI_{a+b-1}$.
\end{lemma}

\begin{proof}
The exact identity
\[
 [g,h]-1
 =g^{-1}h^{-1}\bigl((g-1)(h-1)-(h-1)(g-1)\bigr)
\]
holds in the group ring.  The parenthesized expression belongs to
$\cI_{a+b-2}$.  Since $g^{-1}h^{-1}-1\in\cI_1$, replacing the left unit by
$1$ changes the expression by an element of $\cI_{a+b-1}$.
\end{proof}

With $\cL_s$ placed in degree $s-1$,
\[
 \cL_{<N}:=\bigoplus_{2\le s<N}\cL_s
\]
is a graded Lie subalgebra of the graded $k$-algebra
$\gr_{\cI}(\Delta_{\mathbf Z}(T_N))$.

\section{Congruence symbols and divergence-free fields}

Let
\[
 R=k[x_1,\ldots,x_n],\qquad V=k^n.
\]
Let $\mathrm{GL}(V)$ act rationally on the coordinate ring by
\[
 (a\cdot f)(v)=f(a^{-1}v),
\]
and use the induced action on polynomial vector fields.  Conjugation by a
linear change of variables gives the corresponding action of
$\mathrm{GL}(V)(k)$ on $G$; it preserves $G_s$ and $\cI_a$, and the
leading-field map is equivariant.
For $s\ge1$, put
\[
 W_s=R_s\otimes V=\Sym^s(V^*)\otimes V
\]
and
\[
 \Ld_s=\ker\left(\diver:W_s\to R_{s-1}\right),
 \qquad
 \diver\left(\sum_i f_i\partial_i\right)=\sum_i\partial_i f_i.
\]

\begin{lemma}[Leading field]
\label{lem:leading-field}
If $g\in G_s$, write
\[
 g^*(x_i)=x_i+f_i+O(\m^{s+1}),\qquad f_i\in R_s.
\]
Then
\[
 \ell_s(g)=\sum_i f_i\partial_i
\]
is divergence-free, depends only on $gG_{s+1}$, and gives an injective
linear map
\[
 G_s/G_{s+1}\hookrightarrow\Ld_s.
\]
The group commutator corresponds to the usual Lie bracket of vector fields.
\end{lemma}

\begin{proof}
The degree-$(s-1)$ part of the Jacobian determinant of $g$ is
$\sum_i\partial_i f_i$.  Since $g$ preserves the volume form, this term
vanishes.  Modulo $G_{s+1}$, the group law adds the $f_i$.  To identify the
commutator, let the first nonzero homogeneous terms of two substitutions be
\[
 x\longmapsto x+F(x)+F_{>a}(x),
 \qquad
 x\longmapsto x+H(x)+H_{>b}(x),
\]
where $F$ and $H$ have coefficient degrees $a$ and $b$, respectively, and
the indicated remainders have strictly larger degree.  For a homogeneous
polynomial $H_i$ of degree $b$, the terms in
$H_i(x+F)-H_i(x)$ containing exactly $t$ copies of $F$ have degree
\[
 b-t+ta=b+t(a-1).
\]
The term $t=1$ has degree $a+b-1$, while every term with $t\ge2$ has
strictly larger degree because $a\ge2$.  The same estimate applies to
$F_i(x+H)-F_i(x)$, and any occurrence of $F_{>a}$ or $H_{>b}$ also has
strictly larger degree.  Therefore
\[
 g^*h^*(x_i)-h^*g^*(x_i)
 =\sum_j\bigl(F_j\partial_jH_i-H_j\partial_jF_i\bigr)
   +O(\m^{a+b}).
\]
The inverse substitutions are the identity modulo $\m^2$, so composition
with them does not change the degree-$(a+b-1)$ term.  Thus the leading
commutator is the usual vector-field bracket.  If the leading field of $g$
vanishes, then $g\in G_{s+1}$.
\end{proof}

Fix
\[
 \vol=dx_1\wedge\cdots\wedge dx_n.
\]
Contraction identifies $\Ld_d$ with the closed homogeneous
$(n-1)$-forms of coefficient degree $d$.  Let $\Ex_d\subseteq\Ld_d$ be the
inverse image of the exact forms.

\section{Quadratic brackets and exact forms}

\begin{lemma}[Exactness of divergence-free brackets]
\label{lem:cartan-bracket-exact}
For divergence-free polynomial vector fields $D$ and $E$,
\[
 \iota_{[D,E]}\vol=d(\iota_D\iota_E\vol).
\]
In particular, every bracket of divergence-free fields corresponds to an
exact $(n-1)$-form.
\end{lemma}

\begin{proof}
With the conventions
$\mathcal L_D=d\iota_D+\iota_Dd$ and
$\iota_{[D,E]}=\mathcal L_D\iota_E-\iota_E\mathcal L_D$,
the identities $d\vol=0$ and
$\mathcal L_D\vol=\mathcal L_E\vol=0$, one has
\[
 d(\iota_E\vol)
 =\mathcal L_E\vol-\iota_Ed\vol=0.
\]
Therefore
\[
 \iota_{[D,E]}\vol
 =\mathcal L_D(\iota_E\vol)-\iota_E(\mathcal L_D\vol)
 =(d\iota_D+\iota_Dd)(\iota_E\vol)
 =d(\iota_D\iota_E\vol).
\]
\end{proof}

Thus $[\Ld_2,\Ld_{d-1}]\subseteq\Ex_d$.  The reverse inclusion is the
following theorem.

\begin{theorem}[{\cite[Theorem~A]{MaDerived2026}}]
\label{thm:quadratic-generation}
Let $k$ be a field and let $n\ge3$.  For every $d\ge3$,
\[
 [\Ld_2,\Ld_{d-1}]=\Ex_d.
\]
\end{theorem}

\section{Realization of homogeneous fields}

For the remainder of the paper, let $k=\mathbf F_p$ with $p\ge5$.

\begin{lemma}[Formal inverse]
\label{lem:formal-inverse}
Let $\phi:A=k[[x_1,\ldots,x_n]]\to A$ be a continuous $k$-algebra
endomorphism satisfying
\[
 \phi(x_i)\equiv x_i\pmod{\m^2}
 \qquad(1\le i\le n).
\]
Then $\phi$ is a continuous $k$-algebra automorphism.
\end{lemma}

\begin{proof}
For every $N\ge2$, the induced linear map
$\phi_N:A/\m^N\to A/\m^N$ preserves the $\m$-adic filtration and is the
identity on every associated-graded piece.  Hence
$\phi_N=1+u_N$, where $u_N$ raises filtration and is therefore nilpotent,
so
\[
 \phi_N^{-1}=1-u_N+u_N^2-\cdots
\]
is a finite sum.  Hence $\phi_N$ is bijective, and its inverse is an algebra
homomorphism.  Uniqueness makes the inverses compatible with
truncation.  Passing to the inverse limit gives a continuous algebra
endomorphism $\psi:A\to A$ with $\psi\phi=\phi\psi=1$.
\end{proof}

\begin{lemma}[Quadratic lifts]
\label{lem:quadratic-lifts}
Every field in $\Ld_2$ is the leading field of an element of $G_2$.
\end{lemma}

\begin{proof}
We first describe a spanning set for $\Ld_2$.  Fix $j$.  The
only quadratic monomial fields whose divergence contains $x_j$ are
\[
 x_j^2\partial_j
 \quad\text{and}\quad
 x_jx_i\partial_i\quad(i\ne j),
\]
with divergence coefficients $2$ and $1$, respectively.  The kernel of this
single linear functional is spanned by
\[
 x_j^2\partial_j-2x_jx_i\partial_i\qquad(i\ne j).
\]
Every remaining quadratic monomial field differentiates in a coordinate
absent from its coefficient and is therefore divergence-free.

If $i\notin\{a,b\}$, the monomial shear
\[
 x_i\longmapsto x_i+c x_ax_b,
 \qquad x_j\longmapsto x_j\quad(j\ne i)
\]
has Jacobian determinant one and leading field $c x_ax_b\partial_i$.
Indeed, its Jacobian matrix differs from the identity only in row $i$, and
its $(i,i)$-entry remains $1$ because $x_ax_b$ is independent of $x_i$.
Its linear term is the identity, so it is a formal automorphism by
\Cref{lem:formal-inverse}.

On two coordinates $(u,z)$, the formal maps
\[
 (u,z)\longmapsto\left(\frac{u}{1-cu},\ z(1-cu)^2\right)
\]
are well-defined formal substitutions because $1-cu$ is a unit.  Their Jacobian matrix is triangular with diagonal entries
$(1-cu)^{-2}$ and $(1-cu)^2$, hence has determinant one; the quadratic
leading field is
\[
 c\bigl(u^2\partial_u-2uz\partial_z\bigr).
\]
Their linear term is the identity, so \Cref{lem:formal-inverse} makes
them formal automorphisms.  Permuting coordinates gives all
diagonal-divergence cancellation fields.  Together with the off-diagonal
monomial shears, they span $\Ld_2$.  Products of the corresponding
automorphisms add their leading fields.
\end{proof}

\begin{lemma}[Cartier monomial lifts]
\label{lem:cartier-monomial-lifts}
Let $r\ge0$ and $d=(p-1)(n-1)+pr$.  Every Cartier monomial basis class in
degree $d$ has a representative vector field arising as the leading field of
a pure Jacobian-one shear.  These classes span the degree-$d$ de Rham
cohomology.
\end{lemma}

\begin{proof}
A basis representative has the form
\[
 M_{\alpha,j}
 =x^{p\alpha+(p-1)(\one-e_j)}\partial_j,
 \qquad |\alpha|=r.
\]
The exponent of $x_j$ is $p\alpha_j$, so
\[
 \partial_j\bigl(x^{p\alpha+(p-1)(\one-e_j)}\bigr)=0.
\]
Therefore
\[
 x_j\longmapsto x_j+c x^{p\alpha+(p-1)(\one-e_j)},
 \qquad x_i\longmapsto x_i\quad(i\ne j),
\]
has Jacobian determinant
$1+c\partial_j(x^{p\alpha+(p-1)(\one-e_j)})=1$ and leading field
$cM_{\alpha,j}$.
Its linear term is the identity, so it is a formal automorphism by
\Cref{lem:formal-inverse}.
\end{proof}

\begin{proposition}[Surjectivity of the leading-field map]
\label{prop:leading-surjective}
For every $2\le s<N$, the map in \Cref{lem:leading-field} is an
isomorphism
\[
 G_s/G_{s+1}\cong\Ld_s.
\]
\end{proposition}

\begin{proof}
Injectivity follows from \Cref{lem:leading-field}.  For surjectivity we
argue by induction on $s$.

The case $s=2$ is \Cref{lem:quadratic-lifts}.  Suppose $s\ge3$.  The de
Rham calculation in \Cref{sec:polynomial-de-rham}, specifically
\eqref{eq:de-rham-basis} and \eqref{eq:de-rham-support}, shows that every
class modulo $\Ex_s$ is a linear combination of the Cartier monomials in
\Cref{lem:cartier-monomial-lifts}, each of which has a pure shear lift.  By
\Cref{thm:quadratic-generation}, every element of $\Ex_s$ is a sum of
brackets $[u,v]$ with $u\in\Ld_2$ and $v\in\Ld_{s-1}$.  Lift $u$ and $v$
by the quadratic lemma and the induction hypothesis.  Their group
commutator has leading field $[u,v]$.  Since products add leading fields,
every element of $\Ld_s$ is realized.
\end{proof}

The additive-symbol and leading-field maps give isomorphisms
\[
 \cL_s\cong G_s/G_{s+1}\cong\Ld_s.
\]
By \Cref{lem:bracket-symbol}, these assemble to an isomorphism of graded
Lie algebras
\begin{equation}
\label{eq:graded-lie-isomorphism}
 \bigoplus_{2\le s<N}\cL_s
 \cong
 \bigoplus_{2\le s<N}\Ld_s.
\end{equation}
Here both $\cL_s$ and $\Ld_s$ have graded degree $s-1$.
We endow $\cL_s$ with the rational $\mathrm{GL}(V)$-module structure
transported through the canonical leading-field isomorphism
$\cL_s\xrightarrow{\sim}\Ld_s$.  For
$a\in\mathrm{GL}(V)(k)$ this transported action agrees with the conjugation
action above, by the equivariance of the leading-field map.  With these
structures, \eqref{eq:graded-lie-isomorphism} is an isomorphism of rational
$\mathrm{GL}(V)$-modules in every homogeneous degree.

\section{Polynomial de Rham cohomology}\label{sec:polynomial-de-rham}

For one variable, the de Rham complex is
\[
 0\longrightarrow k[x]
 \xrightarrow{\partial_x}
 k[x]dx
 \longrightarrow0.
\]
In characteristic $p$,
\[
 \ker\partial_x=k[x^p],
\qquad
 \operatorname{coker}\partial_x=k[x^p]x^{p-1}dx.
\]
Indeed, $\partial_x(x^m)=mx^{m-1}$ vanishes exactly when $p\mid m$,
while the monomials omitted from the image are precisely $x^{pa+p-1}dx$.

The $n$-variable de Rham complex is the tensor product of the one-variable
complexes.  In each variable, split the monomial basis into the cohomology
representatives above and the exact pairs
$x^m\mapsto m x^{m-1}dx$ for $p\nmid m$.  Tensoring these splittings gives
\begin{equation}
\label{eq:de-rham-basis}
 H^q_{\dR}(R)
 =k[x_1^p,\ldots,x_n^p]\otimes
 \Lambda^q\Span_k\{x_i^{p-1}dx_i:1\le i\le n\}.
\end{equation}

For $q=n-1$, a monomial basis is
\[
 x^{p\alpha+(p-1)(\one-e_j)}
 dx_1\wedge\cdots\widehat{dx_j}\cdots\wedge dx_n.
\]
Its coefficient degree is
\[
 (p-1)(n-1)+p|\alpha|.
\]
Therefore, with
\[
 d_r=(p-1)(n-1)+pr,
\]
we obtain
\begin{equation}
\label{eq:de-rham-support}
 H^{n-1}_{\dR}(R)_d=0\quad(d\ne d_r),
\end{equation}
and at $d=d_r$ the basis is indexed by a monomial of degree $r$ and a
missing differential $dx_j$.

\begin{lemma}[Inverse Cartier isomorphism]
\label{lem:inverse-cartier-natural}
For every $q\ge0$, define on decomposable tensors
\[
 C^{-1}_q:\Sym(V^*)^{(1)}\otimes
 (\Lambda^qV^*)^{(1)}\longrightarrow H^q_{\dR}(R)
\]
by
\begin{equation}
\label{eq:inverse-cartier-direct}
 C^{-1}_q\bigl(f^{(1)}\otimes
 (\ell_1\wedge\cdots\wedge\ell_q)^{(1)}\bigr)
 =\left[
 f^p\prod_{t=1}^q\ell_t^{p-1}
 d\ell_1\wedge\cdots\wedge d\ell_q
 \right].
\end{equation}
Give the source its Cartier grading: if $f$ is homogeneous of degree $r$
and $\omega\in\Lambda^qV^*$, set
\[
 \deg_C\bigl(f^{(1)}\otimes\omega^{(1)}\bigr)
 =pr+q(p-1).
\]
With this grading, $\bigoplus_q C^{-1}_q$ is an isomorphism of graded
algebras.

Each $C^{-1}_q$ is $\mathrm{GL}(V)$-equivariant.  In particular,
for $\ell=\sum_j a_jx_j$,
\begin{equation}
\label{eq:linear-cartier-naturality}
 [\ell^{p-1}d\ell]
 =\sum_j a_j^p[x_j^{p-1}dx_j].
\end{equation}
\end{lemma}

\begin{proof}
In degree one and polynomial degree $p-1$, the splitting used in
\eqref{eq:de-rham-basis} has cohomology basis
$[x_j^{p-1}dx_j]$.  In the expansion of
\[
 \ell^{p-1}d\ell
 =\left(\sum_i a_ix_i\right)^{p-1}
  \left(\sum_j a_jdx_j\right),
\]
the coefficient of $x_j^{p-1}dx_j$ is $a_j^p$.  Projection to the cohomology summands gives the right-hand side of
\eqref{eq:linear-cartier-naturality}.  The complementary closed component
lies in the contractible tensor summands and is exact.  The same calculation
gives additivity in the Frobenius-twisted linear variable.  Scalar
multiplication is twisted because $(c\ell)^{p-1}d(c\ell)=
c^p\ell^{p-1}d\ell$.

Taking wedge products of \eqref{eq:linear-cartier-naturality} proves that
\eqref{eq:inverse-cartier-direct} is alternating and Frobenius-semilinear
in the $\ell_t$, and multiplication by $f^p$ gives the polynomial factor.
Thus the formula descends to the twisted tensor product and defines a
homomorphism of graded algebras.  On monomial basis elements it sends
\[
 (x^\alpha)^{(1)}\otimes
 (dx_{i_1}\wedge\cdots\wedge dx_{i_q})^{(1)}
\]
to the basis class
\[
 x^{p\alpha}x_{i_1}^{p-1}\cdots x_{i_q}^{p-1}
 dx_{i_1}\wedge\cdots\wedge dx_{i_q}
\]
from \eqref{eq:de-rham-basis}, and is therefore an isomorphism.

The span of the $p$-th powers $x_i^p$ transforms by the entrywise $p$-th
powers of the matrix acting on the $x_i$; this is the Frobenius-twisted
action on $(V^*)^{(1)}$.  Equation \eqref{eq:linear-cartier-naturality}
gives the same action on the classes $[x_i^{p-1}dx_i]$.  Because pullback commutes with multiplication, wedge product, and $d$,
formula \eqref{eq:inverse-cartier-direct} intertwines the ordinary action on
$H^q_{\dR}(R)$ with the Frobenius-twisted actions on both source factors.
The construction commutes with extension of scalars to every commutative
$k$-algebra and is therefore an isomorphism of rational
$\mathrm{GL}(V)$-modules.
\end{proof}

\begin{proposition}[Rational module form]
\label{prop:rational-module}
For every $r\ge0$, the following natural sequence is exact and
$\mathrm{GL}(V)$-equivariant:
\[
 0\longrightarrow\Ex_{d_r}
 \longrightarrow\Ld_{d_r}
 \longrightarrow\delta_1\otimes W_r^{(1)}
 \longrightarrow0,
\]
where
\[
 \delta_1=(\det V^*)^{\otimes(p-1)},
 \qquad W_r=\Sym^r(V^*)\otimes V.
\]
For $d\ne d_r$, one has $\Ld_d=\Ex_d$.
\end{proposition}

\begin{proof}
Contraction with $\vol$ gives a natural equivariant isomorphism
\[
 (\Ld_d/\Ex_d)\otimes\det(V^*)
 \cong H^{n-1}_{\dR}(R)_d.
\]
By \Cref{lem:inverse-cartier-natural}, the degree-$d_r$ term is
$\mathrm{GL}(V)$-equivariantly identified with
\[
 \Sym^r(V^*)^{(1)}\otimes
 \bigl(\Lambda^{n-1}V^*\bigr)^{(1)}.
\]
Using the canonical identity
\[
 \Lambda^{n-1}V^*\cong V\otimes\det(V^*)
\]
and the fact that Frobenius twist sends the determinant character to its
$p$-th power, the right-hand side becomes
\[
 \Sym^r(V^*)^{(1)}\otimes V^{(1)}\otimes
 (\det V^*)^{\otimes p}.
\]
Tensoring with $(\det V^*)^{-1}$ gives
\[
 (\det V^*)^{\otimes(p-1)}\otimes
 \bigl(\Sym^r(V^*)\otimes V\bigr)^{(1)}
 =\delta_1\otimes W_r^{(1)}.
\]
The vanishing in all other degrees follows from
\eqref{eq:de-rham-support}.
\end{proof}

\section{Proof of the Cartier augmentation theorem}

Let $\cX_d\subseteq\cL_d$ be the image of $\Ex_d$ under the isomorphism
\eqref{eq:graded-lie-isomorphism}.

\begin{proof}[Proof of \Cref{thm:cartier-augmentation}]
Under \eqref{eq:graded-lie-isomorphism}, the bracket of group symbols becomes
the polynomial vector-field bracket.  Hence
\[
 [\cL_2,\cL_{d-1}]
 \cong[\Ld_2,\Ld_{d-1}]
 =\Ex_d
 \cong\cX_d
\]
by \Cref{thm:quadratic-generation}.  The quotient is therefore $\Ld_d/\Ex_d$, and
\Cref{prop:rational-module} gives its support and module structure.
Every homogeneous bracket in degree $d$ lies in $\cX_d$ by
\Cref{lem:cartan-bracket-exact}, while the fixed-source bracket already fills
$\cX_d$.  Hence
\[
 [\cL_{<N},\cL_{<N}]_d
 =\sum_{a+b-1=d}[\cL_a,\cL_b]
 =\cX_d,
\]
so the displayed quotient is the full degree-$d$ indecomposable quotient.
\end{proof}

\begin{corollary}[Dimension]
If $d=d_r<N$,
\[
 \dim_k\bigl(\cL_{d_r}/[\cL_2,\cL_{d_r-1}]\bigr)
 =n\binom{n+r-1}{r}.
\]
\end{corollary}

\begin{proof}
The determinant line and Frobenius twist do not alter dimension, and
\[
 \dim\Sym^r(V^*)=\binom{n+r-1}{r}.
\]
\end{proof}

\begin{corollary}[Stability under truncation]
For every $r\ge0$ and $N>d_r$, the degree-$d_r$ indecomposable quotient is
nonzero and unchanged under further truncation.
\end{corollary}

\begin{proof}
Let $N'\ge N>d_r$.  For
$s\in\{2,d_r-1,d_r\}$, truncation identifies the group-symbol space with
the quotient independent of $N$
\[
 \cL_s^N\cong G_s/G_{s+1}\cong\cL_s^{N'}.
\]
These identifications preserve the commutator by
\Cref{lem:bracket-symbol}.  They therefore identify
\[
 [\cL_2^N,\cL_{d_r-1}^N]
 \quad\text{with}\quad
 [\cL_2^{N'},\cL_{d_r-1}^{N'}]
\]
and induce an isomorphism of the two degree-$d_r$ indecomposable quotients.
Nonvanishing follows from \Cref{prop:rational-module}.
\end{proof}

\end{document}